\documentclass{amsart}

\usepackage[latin1]{inputenc}
\usepackage{color}
\usepackage{url}
\usepackage{multicol}
\usepackage{tikz}
\usetikzlibrary{calc,through,backgrounds}

\def\Dfn#1{{\sf #1}}

\def\1#1{\operatorname{\bf 1}_k}
\def\S#1{\mathcal{S}_{#1}}

\def\D{\operatorname{D}}

\def\RP{\mathcal{RP}}

\def\Fne{\mathcal{F}_{N\!E}}
\def\Fse{\mathcal{F}_{S\!E}}
\def\As{\mathcal{A}_\sigma}
\def\FT{\mathcal{FT}}

\def\Cat{\operatorname{Cat}}

\def\CS{\mathcal{CS}}

\def\textcross{
	\begin{minipage}{13pt}
		\begin{tikzpicture}[scale=1]
			\cpipedream{0.4}{(0,0)}{0/0/black/black}
		\end{tikzpicture}
	\end{minipage}
}
\def\textelbow{
	\begin{minipage}{13pt}
		\begin{tikzpicture}[scale=1]
			\tpipedream{0.4}{(0,0)}{0/0/black/black}
		\end{tikzpicture}
	\end{minipage}
}

\newcommand{\ngon}[8]{ 

	\foreach \t in {1,...,#1} {
		\coordinate (#7\t) at ($#2+(#8-\t*360/#1:#3)$);
	}

	\foreach \x/\y/\z in {#4}{
		\draw[\z,shorten <=#5pt, shorten >=#5pt] {(#7\x)--(#7\y)};
	}

	\setcounter{intege}{1}
	\pgfmathsetcounter{intege}{1}
	\foreach \object in {#6}{
		\node[inner sep=0pt] at (#7\theintege) {\object};
		\pgfmathsetcounter{intege}{\theintege+1}
		\setcounter{intege}{\theintege}
	}
}

\newcommand{\tikzbox}[8]{ 
  \coordinate (A) at #2;
  \coordinate (B) at ($ (A) + (#1,0) $);
  \coordinate (C) at ($ (A) + (0,#1) $);
  \coordinate (D) at ($ (A) + (#1,#1)$);
  \coordinate (E) at ($ (A) + (#1/2,#1/2) $);

  \draw[fill=#7] (A) rectangle (D);

  \draw[color=#3] (A) -- (B);
  \draw[color=#6] (A) -- (C);
  \draw[color=#4] (B) -- (D);
  \draw[color=#5] (C) -- (D);
  
  \node at (E) []{#8};
}

\newcommand{\bgbox}[4]{ 
  \tikzbox{#1}{#2}{black}{black}{black}{black}{#3}{#4}
}

\newcommand{\blackbox}[3]{ 
  \tikzbox{#1}{#2}{black}{black}{black}{black}{white}{#3}
}

\newcommand{\boxcollection}[3]{ 
  \coordinate (X) at #2;
	
  \foreach \x/\y/\object in {#3} {
		\blackbox{#1}{($ (X) + (#1 * \x, - #1 * \y) $)}{\object};
	}
}

\newcommand{\bgboxcollection}[3]{ 
  \coordinate (X) at #2;
	
  \foreach \x/\y/\object/\c in {#3} {
		\bgbox{#1}{($ (X) + (#1 * \x, - #1 * \y) $)}{\c}{\object};
	}
}

\newcommand{\content}[3]{ 
  \coordinate (C) at #2;
	
  \foreach \x/\y/\object in {#3} {
		\node at ($ (C) + ( #1 * \x, -#1 * \y ) + ( #1 / 2, -#1 / 2 )$) {\object};
	}
}

\newcommand{\tpipedream}[3]{ 
  \coordinate (P) at #2;
	
  \foreach \x/\y/\a/\b in {#3} {
		\coordinate (P1) at ($ (P) + ( #1 * \x , -#1 * \y ) + ( 0      , #1 / 2 ) $);
	  \coordinate (P2) at ($ (P) + ( #1 * \x , -#1 * \y ) + ( #1     , #1 / 2 ) $);
	  \coordinate (P3) at ($ (P) + ( #1 * \x , -#1 * \y ) + ( #1 / 2 , #1     ) $);
	  \coordinate (P4) at ($ (P) + ( #1 * \x , -#1 * \y ) + ( #1 / 2 , 0 ) $);
	  \coordinate (P5) at ($ (P) + ( #1 * \x , -#1 * \y ) + ( #1 / 2 , #1 / 2 ) $);
	  \draw[rounded corners=4, color=\a, thick] (P1) -- (P5) -- (P3);
	  \draw[rounded corners=4, color=\b, thick] (P4) -- (P5) -- (P2);
	}

}

\newcommand{\cpipedream}[3]{ 
  \coordinate (P) at #2;
	
  \foreach \x/\y/\a/\b in {#3} {
		\coordinate (P1) at ($ (P) + ( #1 * \x , -#1 * \y ) + ( 0      , #1 / 2 ) $);
	  \coordinate (P2) at ($ (P) + ( #1 * \x , -#1 * \y ) + ( #1     , #1 / 2 ) $);
	  \coordinate (P3) at ($ (P) + ( #1 * \x , -#1 * \y ) + ( #1 / 2 , #1     ) $);
	  \coordinate (P4) at ($ (P) + ( #1 * \x , -#1 * \y ) + ( #1 / 2 , 0 ) $);
	  \coordinate (P5) at ($ (P) + ( #1 * \x , -#1 * \y ) + ( #1 / 2 , #1 / 2 ) $);
	  \draw[rounded corners=0.2, color=\a, thick] (P1) -- (P5) -- (P3);
	  \draw[rounded corners=0.2, color=\b, thick] (P4) -- (P5) -- (P2);
	}

}

\newcommand{\latticepath}[4]{ 
  \coordinate (L) at #2;
	
  \foreach \x/\y/\a in {#4} {
  	\coordinate (L1) at ($ (L) + ( #1 * \x , #1 * \y ) $);
	  \draw[color=\a, #3] (L) -- (L1);
	  \coordinate (L) at (L1);
	}

}

\definecolor{verylightgrey}{rgb}{.80,.80,.80}
\definecolor{lightgrey}{rgb}{.6 , .6 , .6}
\definecolor{grey}{rgb}{.5 , .5 , .5}
\definecolor{darkgrey}{rgb}{.35 , .35 , .35}

\newtheorem{theorem}{Theorem}[section]

\newtheorem{corollary}[theorem]{Corollary}
\newtheorem{lemma}[theorem]{Lemma}
\newtheorem{proposition}[theorem]{Proposition}
\newtheorem{conjecture}[theorem]{Conjecture}

\theoremstyle{definition}

\begin{document}
	\title[Maximal fillings, simplicial complexes, and Schubert polynomials]{Maximal fillings of moon polyominoes, simplicial complexes, and Schubert polynomials}

	\author[Luis Serrano]{Luis Serrano$^*$}
	\author[Christian Stump]{Christian Stump$^\dagger$}

	\address{LaCIM, Universit\'e du Qu\'ebec \`a Montr\'eal, Montr\'eal Canada \\}
	\email{serrano@lacim.ca, christian.stump@lacim.ca}

	\date{\today}
	\thanks{$^*$supported by NSERC; $^\dagger$supported by CRM-ISM}
	\begin{abstract}
		We exhibit a canonical connection between maximal $(0,1)$-fillings of a moon polyomino avoiding north-east chains of a given length and reduced pipe dreams of a certain permutation. Following this approach we  show that the simplicial complex of such maximal fillings is a vertex-decomposable, and thus shellable, sphere. In particular, this implies a positivity result for Schubert polynomials. For Ferrers shapes, we moreover construct a bijection to maximal fillings avoiding south-east chains of the same length which specializes to a bijection between $k$-triangulations of the $n$-gon and $k$-fans of Dyck paths of length $2(n-2k)$. Using this, we translate a conjectured cyclic sieving phenomenon for $k$-triangulations with rotation to the language of $k$-flagged tableaux with promotion.
	\end{abstract}
	\maketitle

	\section{Introduction}
		Fix positive integers $n$ and $k$ such that $2k < n$. A \Dfn{$k$-triangulation} of a convex $n$-gon is a maximal collection of diagonals in the $n$-gon such that no $k+1$ diagonals mutually cross. A \Dfn{$k$-fan of Dyck paths} of length $2 \ell$ is a collection of $k$ Dyck paths from $(0,0)$ to $(\ell,\ell)$ which  do not cross (although they may share edges).
		
		The following theorem is the first main result in this article. It answers a question in R.~Stanley's Catalan Addendum \cite{Sta2011}, and extends results by S.~Elizalde~\cite{Eli2007} and C.~Nicol\'as~\cite{Nic2009}.
		\begin{theorem}\label{th:1}
			There is an explicit bijection between $k$-triangulations of a convex $n$-gon and $k$-fans of Dyck paths of length $2(n-2k)$.
		\end{theorem}
		A \Dfn{north-east chain} of length $\ell$ in a Ferrers shape $\lambda$ is a sequence of $\ell$ boxes in $\lambda$ such that every box in the sequence is strictly north and strictly east of the preceding one, and for which the smallest rectangle containing all boxes in the sequence is also contained in $\lambda$. A \Dfn{$k$-north-east filling} of $\lambda$ is a $(0,1)$-filling which does not contain any north-east chain of $1$'s of length $k+1$, and in which the number of $1$'s is maximal. As usual, we identify a $(0,1)$-filling with its set of boxes filled with $1$'s and draw them by marking its set of boxes by $+$'s. See Figure~\ref{fig:pipe dream}(a) for an example. The set of all $k$-north-east fillings of $\lambda$ is denoted by $\Fne(\lambda,k)$. \Dfn{South-east chains}, \Dfn{$k$-south-east fillings} and $\Fse(\lambda,k)$ are defined similarly.

		It is well known that $k$-triangulations of the $n$-gon can be seen as $k$-north-east fillings of the staircase shape $(n-1,\ldots,2,1)$, and furthermore, $k$-fans of Dyck paths of length $2(n-2k)$ can be seen as $k$-south-east fillings of the same staircase (see e.g. \cite{Kra2006, Rub2006}). Thus, the second main theorem is a clear extension of the first. It answers a questions raised by C.~Krattenthaler in \cite{Kra2006}.
		\begin{theorem}\label{th:2}
			Let $\lambda$ be a Ferrers shape and let $k$ be a positive integer. There is an explicit bijection between $k$-north-east and $k$-south-east fillings of $\lambda$.
		\end{theorem}
		The constructed bijection goes through two intermediate objects, namely through pipe dreams and flagged tableaux, both arising in the theory of Schubert polynomials. The third main theorem is a central step in the proof of Theorem~\ref{th:2} and it concerns the connection between north-east chains and reduced pipe dreams.
		\begin{theorem}\label{th:3}
			Let $\lambda$ be a Ferrers shape and let $k$ be a positive integer. There exists a canonical bijection between $k$-north-east fillings of $\lambda$ and reduced pipe dreams of a permutation depending on $\lambda$ and $k$.
		\end{theorem}
		This bijection will be described in Section~\ref{sec:ne to pipe}. A variation of the argument gives the following generalization to moon polyominoes as defined in Section~\ref{sec:moon}.
		\begin{theorem}\label{th:4}
			Let $M$ be a moon polyomino and let $k$ be a positive integer. Then there exists a canonical bijection between $k$-north-east fillings of $M$ and reduced pipe dreams (of a given permutation) living inside~$M$.
		\end{theorem}
		We will use the construction to obtain new properties and simple proofs for known properties of $k$-north-east fillings and of $k$-triangula\-tions. In particular, we obtain the following corollaries.
		\begin{corollary}\label{cor:1}
			The simplicial complex with facets being $k$-north-east fillings of a moon polyomino $M$ is the join of a vertex-decomposable, triangulated sphere with a full simplex. In particular, it is shellable and Cohen-Macaulay.
		\end{corollary}
		\begin{corollary}\label{cor:3}
			Let $S$ be a stack polyomino and $\lambda$ the Ferrers shape obtained from $S$ be properly rearranging its columns. Let $\sigma$ and $\tau$ be the associated permutations. Then the difference
			$$\mathfrak{S}_\sigma(x_1,x_2,\ldots) - \mathfrak{S}_\tau(x_1,x_2,\ldots)$$
			of Schubert polynomials is monomial positive.
		\end{corollary}

		The bijection for $k$-triangulations has the additional property that the cyclic action given by rotation of the $n$-gon corresponds to a promotion-like operation on flagged tableaux and thus transforms a conjectured cyclic sieving phenomenon (CSP) into the context of $k$-flagged tableaux.
		\begin{conjecture}\label{con:CSP2}
			Let $\FT(\lambda,k)$ be the set of $k$-flagged tableaux and let $\rho$ be the promotion-like cyclic action on $\FT(\lambda,k)$. The triple
			$$\Big( \FT(\lambda,k), \langle \rho \rangle, F(q) \Big),$$
			exhibits the CSP, where
				$$F(q) := \prod_{1 \leq i \leq j < n-2k} \frac{[i+j+2k]_q}{[i+j]_q}$$
				is a natural $q$-analogue of the cardinality of $\Fne(\lambda,k)$.
		\end{conjecture}

	\section{From north-east fillings to pipe dreams}\label{sec:ne to pipe}
		In this section we exhibit a connection between $k$-north-east fillings of Ferrers shapes as well as of stack and moon polyominoes on the one hand  and reduced pipe dreams on the other. This generalizes a construction by the second author for $k$-triangulations \cite{Stu2010}, and by Pilaud and Pocchiola~\cite{Pil2010,PP2010}, where they refer to pipe dreams as pseudoline arrangements.

		Reduced pipe dreams (or $rc$-graphs) were introduced by S.~Fomin and A.~Kirillov in \cite{FK1996} (see also work of N.~Bergeron and S.~Billey \cite{BB1993}). They play a central role in the combinatorics of Schubert polynomials of A.~Lascoux and M.-P.~Sch\"utzenberger. A \Dfn{pipe dream} of size $n$ is a filling of the staircase shape $(n-1,\ldots,2,1)$ where each box contains two crossing pipes $\textcross$ or two turning pipes $\textelbow$. See Figure~\ref{fig:pipe dream}(b) for an example. A pipe dream is identified with its set of boxes containing two crossing pipes $\textcross$.
		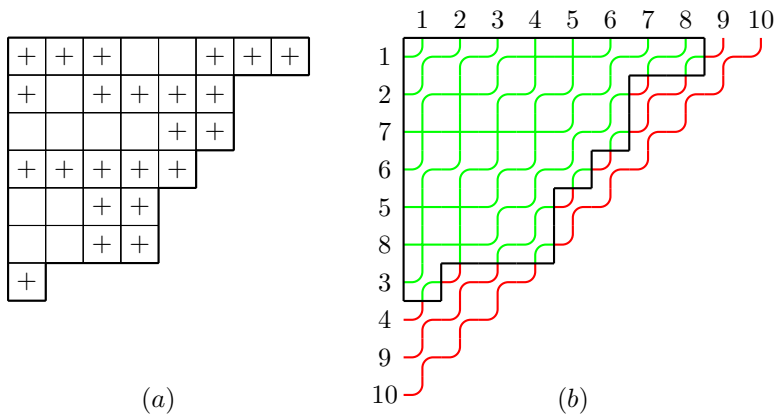
\begin{figure}
			\centering
			$\begin{array}{ccc}
			\begin{tikzpicture}[scale=1]
				\boxcollection{0.5}{(0,0)}{
					0/0/$+$,1/0/$+$,2/0/$+$,3/0/,4/0/,5/0/$+$,6/0/$+$,7/0/$+$,
					0/1/$+$,1/1/,2/1/$+$,3/1/$+$,4/1/$+$,5/1/$+$,
					0/2/,1/2/,2/2/,3/2/,4/2/$+$,5/2/$+$,
					0/3/$+$,1/3/$+$,2/3/$+$,3/3/$+$,4/3/$+$,
					0/4/,1/4/,2/4/$+$,3/4/$+$,
					0/5/,1/5/,2/5/$+$,3/5/$+$,
					0/6/$+$}
				\latticepath{0.5}{(0,-3)}{thick}{
					1/0/black,0/1/black,1/0/black,1/0/black,1/0/black,0/1/black,0/1/black,1/0/black,0/1/black,1/0/black,0/1/black,0/1/black,1/0/black,1/0/black,0/1/black,-8/0/black,0/-7/black}
				\content{0.5}{(0,-4.15)}{
					0/0/}
			\end{tikzpicture}
			&&
			\begin{tikzpicture}[scale=1]
				\tpipedream{0.5}{(0,0)}{
					0/0/green/green,1/0/green/green,2/0/green/green,5/0/green/green,6/0/green/green,7/0/green/green,8/0/red/red,9/0/red/white,
					0/1/green/green,2/1/green/green,3/1/green/green,4/1/green/green,5/1/green/green,6/1/red/red,7/1/red/red,8/1/red/white,
					4/2/green/green,5/2/green/green,6/2/red/red,7/2/red/white,
					0/3/green/green,1/3/green/green,2/3/green/green,3/3/green/green,4/3/green/green,5/3/red/red,6/3/red/white,
					2/4/green/green,3/4/green/green,4/4/red/red,5/4/red/white,
					2/5/green/green,3/5/green/green,4/5/red/white,
					0/6/green/green,1/6/red/red,2/6/red/red,3/6/red/white,
					0/7/red/red,1/7/red/red,2/7/red/white,
					0/8/red/red,1/8/red/white,
					0/9/red/white}
				\cpipedream{0.5}{(0,0)}{
					3/0/green/green,4/0/green/green,
					1/1/green/green,
					0/2/green/green,1/2/green/green,2/2/green/green,3/2/green/green,
					0/4/green/green,1/4/green/green,
					0/5/green/green,1/5/green/green}
				\latticepath{0.5}{(0,-3)}{thick}{
					1/0/black,0/1/black,1/0/black,1/0/black,1/0/black,0/1/black,0/1/black,1/0/black,0/1/black,1/0/black,0/1/black,0/1/black,1/0/black,1/0/black,0/1/black,-8/0/black,0/-7/black}
				\content{0.5}{(0,1)}{
					0/0/$1$,1/0/$2$,2/0/$3$,3/0/$4$,4/0/$5$,5/0/$6$,6/0/$7$,7/0/$8$,8/0/$9$,9/0/$10$,
					-1/1/$1$,-1/2/$2$,-1/3/$7$,-1/4/$6$,-1/5/$5$,-1/6/$8$,-1/7/$3$,-1/8/$4$,-1/9/$9$,-1/10/$10$}
			\end{tikzpicture}
			\\[-15pt]
			(a) && (b)
			\end{array}$
			\caption{A $2$-north-east filling of $\lambda = (8,6,6,5,4,4,1)$ and its associated reduced pipe dream.}
			\label{fig:pipe dream}
		\end{figure}
		The permutation $\pi(D)$ of a pipe dream $D$ is obtained by following the pipes starting from the top and going all the way to the left, and then reading $\pi(D)$ on the left from top to bottom in one line notation. For example, the permutation of the pipe dream in Figure~\ref{fig:pipe dream}(b) is $[1,2,7,6,5,8,3,4,9,10]$. A pipe dream is \Dfn{reduced} if two pipes cross at most once. We say that a pipe dream \Dfn{lives inside} a set $M$ of boxes in the staircase shape if all its crossings are contained in $M$. For a given permutation $\pi$ and a set $M$ of boxes, denote the set of reduced pipe dreams for $\pi$ by $\RP(\pi)$ and the set of reduced pipe dreams for $\pi$ which live inside $M$ by $\RP(\pi,M)$.

		\subsection{A bijection between north-east fillings and reduced pipe dreams}
		Starting with a $k$-north-east filling of $\lambda$, one obtains a pipe dream by replacing every $1$ by two turning pipes and every $0$ by two crossing pipes. Afterwards, $\lambda$ is embedded into the smallest staircase containing it, and all boxes in the staircase outside of $\lambda$ are replaced by turning pipes. In other words, a $k$-north-east filling of $\lambda$ and its associated pipe dream are complementary $(0,1)$-fillings of $\lambda$ when both are identified with their sets of boxes. For example, the $\textcross$'s in the pipe dream in Figure~\ref{fig:pipe dream}(b) and the marked boxes in (a) are complementary $(0,1)$-fillings of $\lambda$. The pieces in boxes outside of $\lambda$ are drawn in the pipe dream in {\color{red}red} whereas pieces within $\lambda$ are drawn in {\color{green}green}. We call this identification between $k$-north-east fillings of $\lambda$ and reduced pipe dreams \Dfn{complementary map}.
		
		For a permutation $\sigma \in \S{n}$, define its (\Dfn{Rothe}) \Dfn{diagram} (see \cite[Section~2.1]{Man2001}) to be the set of boxes in the staircase shape given by
		$$\D(\sigma) := \big\{ (i,\sigma_j) : i < j, \sigma_i > \sigma_j \big\}.$$
		For example, the diagram of $[1,2,7,6,5,8,3,4,9,10]$ in Figure~\ref{fig:permutation} is given by the shaded area. Clearly, the number of boxes in $\D(\sigma)$ equals the \Dfn{length} of $\sigma$, i.e., the minimal number of simple transpositions needed to write $\sigma$. A permutation is called \Dfn{dominant} if its diagram is a Ferrers shape containing the box $(1,1)$. By construction, different permutations in $\S{n}$ have different shapes and one can obtain every Ferrers shape in this way for some $n$. Thus, starting with a Ferrers shape $\lambda$, let $\sigma(\lambda)$ be the unique dominant permutation $\sigma \in \S{n}$ for which $\D(\sigma) = \lambda$, where $n$ is given by the size of the smallest staircase shape containing $\lambda$. Moreover, define $\sigma_k(\lambda)$ to be
		$$\1{k} \times \tau := [1,2,\ldots,k, \tau_1+k, \ldots, \tau_n+k] \in \S{n+k}$$
		where $\tau = \sigma(\mu)$ and $\mu$ is obtained from $\lambda$ by removing its first $k$ rows and columns. Graphically, this means that $\sigma_k(\lambda)$ is obtained by removing the first $k$ columns and rows from $\sigma(\lambda)$. Note that the north-west corner of $\sigma_k(\lambda)$ remains in box $(k+1,k+1)$. See Figure~\ref{fig:permutation} for $\sigma_2(\lambda)$ with $\lambda$ as in Figure~\ref{fig:pipe dream}.
		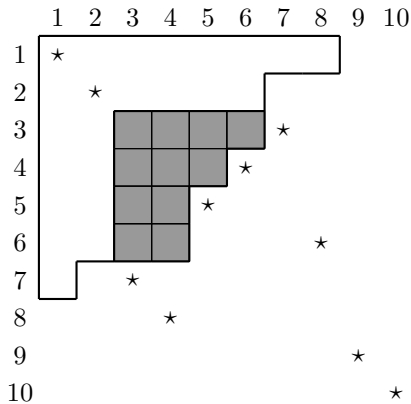
\begin{figure}
			\begin{tikzpicture}[scale=1]
				\bgboxcollection{0.5}{(0,0)}{
					2/2//lightgrey,3/2//lightgrey,4/2//lightgrey,5/2//lightgrey,
					2/3//lightgrey,3/3//lightgrey,4/3//lightgrey,
					2/4//lightgrey,3/4//lightgrey,
					2/5//lightgrey,3/5//lightgrey}
				\latticepath{0.5}{(0,-3)}{thick}{
					1/0/black,0/1/black,1/0/black,1/0/black,1/0/black,0/1/black,0/1/black,1/0/black,0/1/black,1/0/black,0/1/black,0/1/black,1/0/black,1/0/black,0/1/black,-8/0/black,0/-7/black}
				\latticepath{0.5}{(1,-2.5)}{thick}{
					0/1/black,0/1/black,0/1/black,0/1/black,1/0/black,1/0/black,1/0/black,1/0/black}
				\content{0.5}{(0,1)}{
					0/0/$1$,1/0/$2$,2/0/$3$,3/0/$4$,4/0/$5$,5/0/$6$,6/0/$7$,7/0/$8$,8/0/$9$,9/0/$10$,
					-1/1/$1$,-1/2/$2$,-1/3/$3$,-1/4/$4$,-1/5/$5$,-1/6/$6$,-1/7/$7$,-1/8/$8$,-1/9/$9$,-1/10/$10$,
					0/1/$\star$,1/2/$\star$,6/3/$\star$,5/4/$\star$,4/5/$\star$,7/6/$\star$,2/7/$\star$,3/8/$\star$,8/9/$\star$,9/10/$\star$}
			\end{tikzpicture}
			\caption{$\sigma_2(\lambda) = [1,2,7,6,5,8,3,4,9,10]$ for $\lambda = (8,6,6,5,4,4,1)$.}
			\label{fig:permutation}
		\end{figure}
		The following theorem is a more precise reformulation of Theorem~\ref{th:3}.
		\begin{theorem}\label{th:fillings-pipe dreams}
			Let $\lambda$ be a Ferrers shape and let $\sigma = \sigma_k(\lambda)$. The complementary map from $k$-north-east fillings to pipe dreams is a bijection between $\Fne(\lambda,k)$ and~$\RP(\sigma)$.
		\end{theorem}
		For the proof of this theorem we use an alternative description of pipe dreams as given by A.~Knutson and E.~Miller in \cite[Theorem~B]{KM2005}, or, in a more combinatorial language, by N.~Jia and E.~Miller in \cite[Theorem~3]{JM2008}. First, we observe that the definition of antidiagonals in \cite[Definition~2]{JM2008} is equivalent to the definition of a north-east chain inside $[n] \times [n]$. Following the notion in the latter, define $\As$ for $\sigma \in \S{n}$ to be the collection over all $1 \leq p,q \leq n$ of all minimal north-east chains of length $r_\sigma(p,q) + 1$ lying inside the rectangle $[p] \times [q]$ where
		$$r_\sigma(p,q) := \#\big\{ (i,j) : i \leq p, j \leq q, \sigma(i) = j \big\}.$$
		As it can be seen in Figure~\ref{fig:permutation}, $r_\sigma(p,q)$ equals the number of stars in the matrix presentation of $\sigma$ lying inside the rectangle $[p]\times[q]$. 

		The set $\RP(\sigma)$ of reduced pipe dreams for $\sigma$ can be described in terms of $\As$ as follows. A subset of the staircase $(n-1,\ldots,2,1)$ is a reduced pipe dream for $\sigma$ if and only if it intersects every north-east chain in $\As$, and it is minimal in this sense. Looking at this observation in a slightly different way, we obtain the following proposition describing pipe dreams in terms of maximal fillings of the staircase shape.
		\begin{proposition}\label{pro:pipe dream description}
			Let $\sigma \in \S{n}$. A subset of the staircase $(n-1,\ldots,2,1)$ is a reduced pipe dream for $\sigma$ if and only if its complement is a maximal filling not containing any north-east chain in $\As$.
		\end{proposition}
		\begin{proof}
			From the description of reduced pipe dreams above, it follows that a subset $R$ of the staircase shape is a reduced pipe dream for $\sigma_k(\lambda)$ if and only if it intersects every north-east chain in $\As$, and it is minimal with respect to this property. Thus, $R$ is a reduced pipe dream if and only if its complement $R^c$ in the staircase shape does not contain any north-east chain in $\As$, and $R^c$ is maximal with respect to this property. The latter is precisely a maximal filling of the staircase shape not containing any north-east chain in $\As$.
		\end{proof}
		However, we do not need to consider all rectangles $[p] \times [q]$ to define $\As$. It is enough to consider the collection of all south-east corner boxes $(p,q)$ of $\D(\sigma)$, each labelled by $r_\sigma(p,q)$. This labelled collection is called the \Dfn{essential set} of $\sigma$ in \cite[Section~2.2]{Man2001}. See Figure~\ref{fig:moon pipe} for an example.
		\begin{proposition}\label{prop:essential}
			$\As$ is given by the collection of north-east chains of length $r_\sigma(p,q) + 1$ lying inside rectangles $[p] \times [q]$ for boxes $(p,q)$ in the essential set of $\sigma$.
		\end{proposition}
		\begin{proof}
			Every $[i] \times [i]$ rectangle for $i \leq k$, where $k$ is the smallest label of an element in the essential set of $\sigma$, contains $i$ stars. Therefore, $\As$ does not contain any north-east chains of length smaller or equal to $k$. Moreover, observe that $r_\sigma(p,q)$ is constant inside a component of $\D(\sigma)$ and thus, among those it is enough to consider boxes $(p,q)$ in the essential set. As $r_\sigma(p+a,q+b) = r_\sigma(p,q) + a + b$ for such a box $(p,q)$, the lemma follows from the minimality condition in the definition of $\As$.
		\end{proof}
		Using this proposition, we also obtain the following description of $\As$ coming from Ferrers shapes.
		\begin{proposition}\label{prop:As}
			Let $\lambda$ be a Ferrers shape and let $\sigma := \sigma_k(\lambda)$. $\As$ is given by the collection of all north-east chains of length $k+1$ in $\lambda$.
		\end{proposition}
		\begin{proof}
			We have already seen that the essential set of $\sigma$ is given by the south-east corner boxes of maximal rectangles in $\lambda$ of width and height strictly larger than $k$. As all those maximal rectangles are of the form $[p] \times [q]$, the result follows with the observation that $r_\sigma(q,p) = k$ for such $(p,q)$.
		\end{proof}
		This proposition implies the following well known corollary.
		\begin{corollary}\label{cor:living dream}
			Every reduced pipe dream for $\sigma = \sigma_k(\lambda)$ lives inside $\lambda$, namely
			$$\RP(\sigma) = \RP(\sigma,\lambda).$$
		\end{corollary}

		Putting the arguments together, we can now prove Theorem~\ref{th:fillings-pipe dreams}.
		\begin{proof}[Proof of Theorem~\ref{th:fillings-pipe dreams}]
			The result follows from Propositions~\ref{pro:pipe dream description} and~\ref{prop:As}.
		\end{proof}
		
		\subsection{Generalizations to moon polyominoes}\label{sec:moon}
			The results in the previous section can be partially generalized to moon polyominoes which were studied by J.~Jonsson in \cite{J2005}. A polyomino $M$ (i.e., a set of boxes in the positive integer quadrant) is called \Dfn{convex} if for any two boxes in $M$ lying in the same row or column, all boxes in between are also contained in $M$. Moreover, $M$ is called \Dfn{intersection-free} if for any two columns (or equivalently, rows) of $M$, one is contained in the other. A polyomino is called a \Dfn{moon polyomino} if it is convex and intersection-free.
			\begin{figure}
				\begin{tikzpicture}[scale=1]
					\boxcollection{0.5}{(0,0)}{
						2/0/$+$,3/0/$+$,4/0/,
						0/1/$+$,1/1/$+$,2/1/,3/1/$+$,4/1/,5/1/,
						0/2/,1/2/,2/2/,3/2/$+$,4/2/$+$,5/2/,6/2/,
						0/3/,1/3/,2/3/,3/3/,4/3/$+$,5/3/$+$,6/3/$+$,
						1/4/,2/4/,3/4/,4/4/$+$,
						2/5/,3/5/$+$}
					\latticepath{0.5}{(1,0.5)}{thick}{
						3/0/black,0/-1/black,1/0/black,0/-1/black,1/0/black,0/-2/black,-2/0/black,0/-1/black,-1/0/black,0/-1/black,-2/0/black,0/1/black,-1/0/black,0/1/black,-1/0/black,0/3/black,2/0/black,0/1/black}
					\latticepath{0.5}{(1.05,0.45)}{thick}{
						2.8/0/verylightgrey,0/-4.8/verylightgrey,-2.8/0/verylightgrey,0/4.8/verylightgrey}
					\content{0.5}{(0,-4.65)}{
						0/0/}
				\end{tikzpicture}
				\quad\qquad
				\begin{tikzpicture}[scale=1]
					\tpipedream{0.5}{(0,0)}{
						0/0/red/red,1/0/red/red,2/0/green/green,3/0/green/green,5/0/red/red,6/0/red/red,7/0/red/red,8/0/red/red,9/0/red/red,10/0/red/white,
						0/1/green/green,1/1/green/green,3/1/green/green,6/1/red/red,7/1/red/red,8/1/red/red,9/1/red/white,
						3/2/green/green,4/2/green/green,7/2/red/red,8/2/red/white,
						4/3/green/green,5/3/green/green,6/3/green/green,7/3/red/white,
						4/4/green/green,0/4/red/red,5/4/red/red,6/4/red/white,
						3/5/green/green,0/5/red/red,1/5/red/red,4/5/red/red,5/5/red/white,
						0/6/red/red,1/6/red/red,2/6/red/red,3/6/red/red,4/6/red/white,
						0/7/red/red,1/7/red/red,2/7/red/red,3/7/red/white,
						0/8/red/red,1/8/red/red,2/8/red/white,
						0/9/red/red,1/9/red/white,
						0/10/red/white}
					\cpipedream{0.5}{(0,0)}{
						4/0/green/green,
						2/1/green/green,4/1/green/green,5/1/green/green,
						0/2/green/green,1/2/green/green,2/2/green/green,5/2/green/green,6/2/green/green,
						0/3/green/green,1/3/green/green,2/3/green/green,3/3/green/green,
						1/4/green/green,2/4/green/green,3/4/green/green,
						2/5/green/green}
					\content{0.5}{(0,1)}{
						0/0/$1$,1/0/$2$,2/0/$3$,3/0/$4$,4/0/$5$,5/0/$6$,6/0/$7$,7/0/$8$,8/0/$9$,9/0/$10$,10/0/$11$,
						-1/1/$1$,-1/2/$2$,-1/3/$8$,-1/4/$10$,-1/5/$3$,-1/6/$7$,-1/7/$6$,-1/8/$5$,-1/9/$4$,-1/10/$9$,-1/11/$11$}
					\content{0.5}{(0,-2.6)}{
						0/0/}

					\latticepath{0.5}{(1,0.5)}{thick}{
						3/0/black,0/-1/black,1/0/black,0/-1/black,1/0/black,0/-2/black,-2/0/black,0/-1/black,-1/0/black,0/-1/black,-2/0/black,0/1/black,-1/0/black,0/1/black,-1/0/black,0/3/black,2/0/black,0/1/black}
				\end{tikzpicture}
				\caption{A $1$-north-east filling of a moon polyomino and its associated pipe dream.}
				\label{fig:moon}
			\end{figure}
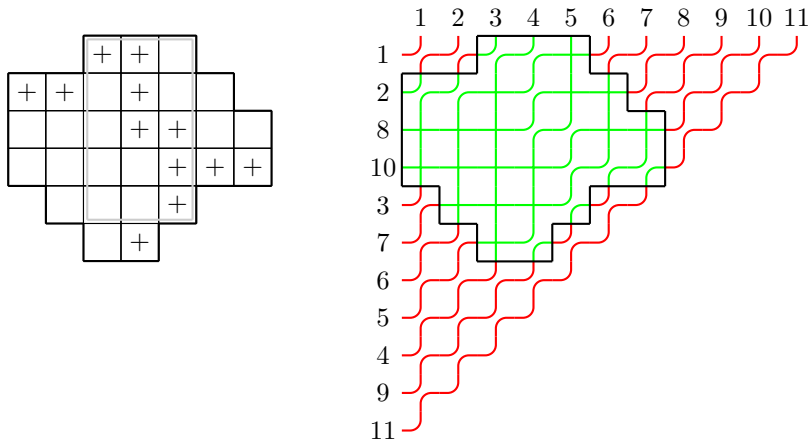
			
			Without loss of generality we consider always moon polyominoes which are north-west justified, namely, they contain boxes both in the first row and in the first column. Observe that Ferrers shapes are special types of moon polyominoes. A $k$-filling of a moon polyomino is defined exactly in the same way as for a Ferrers shape. See  Figure~\ref{fig:moon} for an example.

			To connect $k$-north-east fillings of a moon polyomino $M$ and pipe dreams of a certain permutation $\sigma = \sigma_k(M)$, we must relate maximal fillings of $M$ which do not contain a $(k+1)$-north-east chain in one of its maximal rectangles and maximal fillings of the staircase $(n-1,\ldots,2,1)$ which do not contain a north-east chain of length $r_\sigma(p,q) + 1$ in any rectangle $[p] \times [q]$. Define $\sigma_k(M)$ as follows: for a maximal rectangle $R$ in $M$ of width and height both strictly larger than $k$, let $(a+1,b+1)$ and $(i,j)$ be its north-west and south-east corner boxes. Mark the box $(i+b,j+a)$ with $a+b+k$. $\sigma_k(M)$ is the permutation with this collection as its essential set. This means that the diagram $\D(\sigma)$ of $\sigma$ has $(i+b,j+a)$ as a south-east corner with labels $r_\sigma(i+b,j+a)$ = $a+b+k$. Using \cite[2.2.8]{Man2001}, it is easy to see that this construction is well defined. Note that maximal rectangles of width or height less than or equal to $k$ cannot contain any north-east chain of length larger than $k$ and thus do not contribute to the essential set of the corresponding permutation.
			For example, the moon polyomino $M$ in Figure~\ref{fig:moon}(a) has maximal rectangles $(a+1,b+1) - (i,j)$ given by
			\begin{eqnarray*}
				&(1,3) - (5,5), \quad (1,3) - (6,4),&\\
				&(3,1) - (4,7), \quad (2,2) - (5,5),&\\
				&(2,1) - (4,6),&
			\end{eqnarray*}
			where the first maximal rectangle is highlighted. Thus, for $k = 1$, the resulting essential set and the associated diagram can be seen in Figure~\ref{fig:moon pipe} and the associated permutation is $\sigma_1(M) = [1,2,8,10,3,7,6,5,4,9]$.
			\begin{figure}
				\begin{tikzpicture}[scale=1]
					\content{0.5}{(0,1)}{
					0/0/$1$,1/0/$2$,2/0/$3$,3/0/$4$,4/0/$5$,5/0/$6$,6/0/$7$,7/0/$8$,8/0/$9$,9/0/$10$,
					-1/1/$1$,-1/2/$2$,-1/3/$3$,-1/4/$4$,-1/5/$5$,-1/6/$6$,-1/7/$7$,-1/8/$8$,-1/9/$9$,-1/10/$10$,
					3/8/$3$,4/7/$3$,5/6/$3$,
					8/4/$3$,6/4/$2$,
					0/1/$\star$,1/2/$\star$,2/5/$\star$,3/9/$\star$,4/8/$\star$,5/7/$\star$,
					6/6/$\star$,7/3/$\star$,8/10/$\star$,9/4/$\star$}
					
					\latticepath{0.5}{(1,-0.5)}{thick}{
						5/0/black,0/-2/black,-5/0/black,0/2/black}
					\latticepath{0.5}{(1.5,-2)}{thick}{
						3/0/black,0/-1/black,-1/0/black,0/-1/black,-1/0/black,0/-1/black,-1/0/black,0/3/black}
					\latticepath{0.5}{(4,-1)}{thick}{
						1/0/black,0/-1/black,-1/0/black,0/1/black}
				\end{tikzpicture}
				\caption{The essential set and the diagram $\D(\sigma)$ for $\sigma = [1,2,8,10,3,7,6,5,4,9]$.}
				\label{fig:moon pipe}
			\end{figure}
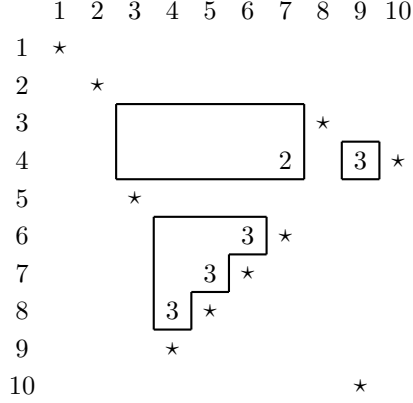
		As all maximal rectangles in a Ferrers shape are of the form $[p] \times [q]$, the definition of $\sigma_k(\lambda)$ reduces in this case to the definition given in the previous section. Moreover observe that in the more general context of moon polyominoes which are not Ferrers shapes, Corollary~\ref{cor:living dream} does not hold.
		
		The following theorem is a more precise reformulation of Theorem~\ref{th:4}.
		\begin{theorem}\label{th:moon fillings-pipe dreams}
			The complementary map from $k$-north-east fillings of a moon polyomino $M$ to pipe dreams of $\sigma = \sigma_k(M)$ is a bijection between $\Fne(M,k)$ and $\RP( \sigma, M )$.
		\end{theorem}
		\begin{proof}
			Recall that the set $\mathcal{A}_\sigma$ is the collection over all $(p,q)$ in the essential set of $\sigma$ of all minimal north-east chains in $[p] \times [q]$ of length $r_\sigma(p,q) + 1$. By construction, every such $(p,q)$ comes from a maximal rectangle $R$ in $M$ with north-west corner $(a+1,b+1)$ and south-east corner $(i,j)$. Thus, $(p,q) = (i+b,j+a)$ and $r_\sigma(p,q) = a+b+k$.
			
			As $M$ is intersection-free by definition, no box strictly south-west or strictly north-east of $R$ is contained in $M$. Therefore, any $(k+1)$-north-east chain inside $R$ can be extended to a $(a+b+k+1)$-north-east chain inside $[p] \times [q]$, compare the maximal rectangle highlighted in Figure~\ref{fig:moon}. This implies that a $(k+1)$-north-east chain inside $R$ cannot be contained in the complement of a pipe dream for $\sigma$ living inside $M$. In total, we obtain that the set of complements of pipe dreams for $\sigma$ living inside $M$ are exactly maximal fillings of $M$ not containing a north-east chain of length $k+1$. This completes the proof.
		\end{proof}

		We now use this theorem together with the main theorem in \cite{J2005} to get new insights on pipe dreams. A \Dfn{stack polyomino} is a moon polyomino where every column starts in the first row. Let $S$ be a stack polyomino and let $\lambda$ be the Ferrers shape obtained from $S$ by properly rearranging the columns. J.~Jonsson proved in \cite[Theorem~14]{J2005} that the number of $k$-north-east fillings of $S$ with a given number of $+$'s in every row equals the number of $k$-north-east fillings in $\lambda$ with the same number of $+$'s in every row. Moreover, he conjectured that this property still holds if the stack polyomino $S$ is replaced by a moon polyomino. Therefore, we obtain the following corollary and the conjecture for the analogous statement for moon polyominoes.
		\begin{corollary}\label{cor:2}
			Let $S$ be a stack polyomino and let $\lambda$ be the associated Ferrers shape. The number of pipe dreams in $\RP( \sigma_k(S), S )$ with a given number of crossings in every row is equal to the number of pipe dreams in $\RP( \sigma_k(\lambda) )$ with the same number of crossings in every row.
		\end{corollary}

		\subsection{The simplicial complex of north-east-fillings}
		We are now in position to prove Corollary~\ref{cor:1}. The canonical connection between $k$-north-east fillings and reduced pipe dreams can be used in the same way as described in the proof of \cite[Corollary~1.3]{Stu2010} for $k$-triangulations in this more general setting. For the necessary background on simplicial complexes and in particular on subword complexes, we refer to \cite{KM2004}. A box in a moon polyomino $M$ is called \Dfn{passive} if it is not contained in any north-east chain in $M$ of length $k+1$. Let $\Delta(M,k)$ be the simplicial complex with vertices being the collection of boxes in $M$, and with facets being $k$-north-east fillings of $M$.
		\begin{corollary}
			$\Delta(M,k)$ is the join of a vertex-decomposable, triangulated sphere and a full simplex of dimension $i-1$, where $i$ equals the number of passive boxes in $M$. In particular, it is shellable and Cohen-Macauley.
		\end{corollary}
		\begin{proof}
			Label the box $(i,j)$ by $i+j-1$. The simplicial complex $\Delta(M,k)$ is precisely the subword complex for the permutation $\sigma_k(M)$ and the word given by the labels of all boxes in $M$ (where $i$ and the simple transposition $s_i$ are identified) read row by row from east to west and from north to south. Observe that the passive boxes are exactly those boxes which are contained in all facets of $\Delta(M,k)$. Thus, the corollary follows from Theorem~\ref{th:fillings-pipe dreams} together with Theorems~2.5 and 3.7 in \cite{KM2004}.
		\end{proof}

		\subsection{A mutation-like operation on pipe dreams}\label{sec:mutation}
		Generalizing the notion in the previous section, one can define a pure simplicial complex $\Delta(\sigma)$ for any $\sigma \in \S{n}$ by defining the facets as the complements in the staircase of reduced pipe dreams in $\RP(\sigma)$ (see \cite{KM2004}). Using the property that two pipes in a reduced pipe dream $D$ cross at most once, one can define a mutation-like operation on facets of $\Delta(\sigma)$ as follows. One can mutate the facet $F(D)$ of $\Delta(\sigma)$ associated to $D$ at a vertex $b$ if the two pipes in $D$ which touch in $b$ cross somewhere else. In other words, one can mutate $F(D)$ at a vertex $b$ if the starting points $i < j$ of the two pipes in $D$ which touch in $b$ form an inversion of $\sigma$. The mutation of $F(D)$ at such a vertex $b$ is then defined to be the facet $F(D')$ for the reduced pipe dream $D'$ such that
		\begin{itemize}
			\item[(i)] the two turning pipes in $b$ are replaced in $D'$ by two crossing pipes,
			\item[(ii)] the unique crossing $b'$ of those two pipes is replaced in $D'$ by two turning pipes.
		\end{itemize}
		By construction, the pipe dream $D' = (D \cup b) \setminus b'$ is again in $\RP(\sigma)$ and thus its complement $F(D') = (F(D) \setminus b) \cup b'$ forms another facet of $\Delta(\sigma)$.
	\section{From pipe dreams to south-east fillings}

		In this section we describe a bijection between pipe dreams for $\sigma_k(\lambda)$ and $k$-south-east fillings of $\lambda$, for a Ferrers shape $\lambda$. For the sake of readability, we do this construction in several steps. A similar approach was described by Fomin--Kirillov \cite{FK1997}, and in the particular case of the permutation $[1,n,n-1,\ldots,3,2]$ by Woo \cite{Woo2004}.

		\subsection{From pipe dreams to flagged tableaux}\label{sec:flagged tableaux}

		Define a \Dfn{$k$-flagged tableau} as a semistandard tableau in which the entries in the $i$-th row are smaller than or equal to $i+k$, and denote the set of $k$-flagged tableaux of shape $\lambda$ by $\FT(\lambda,k)$. These were introduced by M.~Wachs \cite{Wac1985}, where she proves that the Schubert polynomial of a \Dfn{vexilliary permutation} is equal to a \Dfn{flagged Schur function} (see also Reiner--Shimozono \cite[Theorem 24]{RS1995}). We now present a bijection between the set $\RP(\sigma)$ of reduced pipe dreams of $\sigma = \sigma_k(\lambda)$ and the set $\FT(\mu,k)$ of $k$-flagged tableaux of shape $\mu = \D( \sigma)$. This bijection can be found in more generality in the work of C.~Lenart \cite[Section 4]{Len2004} for vexilliary permutations, but we include the full description in this particular case for the sake of completeness. For more on flagged tableaux and their connections to geometry, see, e.g., \cite{KMY2005}.

		For a reduced pipe dream $D \in \RP(\sigma)$ with $\sigma$ being of length $\ell$, define the \Dfn{reading biword} to be the $2 \times \ell$ array by reading $\binom{i}{i+j-1}$ for every crossing box $(i,j)$ in $D$ row by row from east to west and from north to south. See Figure~\ref{fig:biword} for an example.
		\begin{figure}
			$\begin{array}{ccc}
			\begin{minipage}{110pt}
			\begin{tikzpicture}[scale=1]
				\boxcollection{0.5}{(0,0)}{
					0/0/,1/0/,2/0/,3/0/$4$,4/0/$5$,5/0/,6/0/,7/0/,
					0/1/,1/1/$3$,2/1/,3/1/,4/1/,5/1/,
					0/2/$3$,1/2/$4$,2/2/$5$,3/2/$6$,4/2/,5/2/,
					0/3/,1/3/,2/3/,3/3/,4/3/,
					0/4/$5$,1/4/$6$,2/4/,3/4/,
					0/5/$6$,1/5/$7$,2/5/,3/5/,
					0/6/}
				\latticepath{0.5}{(0,-3)}{thick}{
					1/0/black,0/1/black,1/0/black,1/0/black,1/0/black,0/1/black,0/1/black,1/0/black,0/1/black,1/0/black,0/1/black,0/1/black,1/0/black,1/0/black,0/1/black,-8/0/black,0/-7/black}
				\content{0.5}{(0,-3.25)}{
					0/0/}
			\end{tikzpicture}
			\end{minipage}
			&&
			\begin{minipage}{220pt}
				\vspace*{-15pt}
				\[
				\left(
				\begin{array}{cccccccccccccccccccccc}
				1&1&2&3&3&3&3&5&5&6&6 \\
				5&4&3&6&5&4&3&6&5&7&6
				\end{array}
				\right)
				\]
			\begin{tikzpicture}[scale=1]
				\content{0.5}{(1.5,2)}{
					0/0/}
				\boxcollection{0.5}{(2,2)}{
					2/2/$3$,3/2/$4$,4/2/$5$,5/2/$6$,
					2/3/$4$,3/3/$5$,4/3/$6$,
					2/4/$5$,3/4/$6$,
					2/5/$6$,3/5/$7$}
				\latticepath{0.5}{(3,-0.5)}{thick}{
					1/0/black,1/0/black,0/1/black,0/1/black,1/0/black,0/1/black,1/0/black,0/1/black,-4/0/black,0/-4/black}
				\boxcollection{0.5}{(5,2)}{
					2/2/$1$,3/2/$1$,4/2/$2$,5/2/$3$,
					2/3/$3$,3/3/$3$,4/3/$3$,
					2/4/$5$,3/4/$5$,
					2/5/$6$,3/5/$6$}
				\latticepath{0.5}{(6,-0.5)}{thick}{
					1/0/black,1/0/black,0/1/black,0/1/black,1/0/black,0/1/black,1/0/black,0/1/black,-4/0/black,0/-4/black}
			\end{tikzpicture}
			\end{minipage}
			\end{array}$
			\caption{Labelling of the crossing boxes in the pipe dream in Figure~\ref{fig:pipe dream}(b), the corresponding compatible sequence, and its insertion and recording tableau.}
			\label{fig:biword}
		\end{figure}
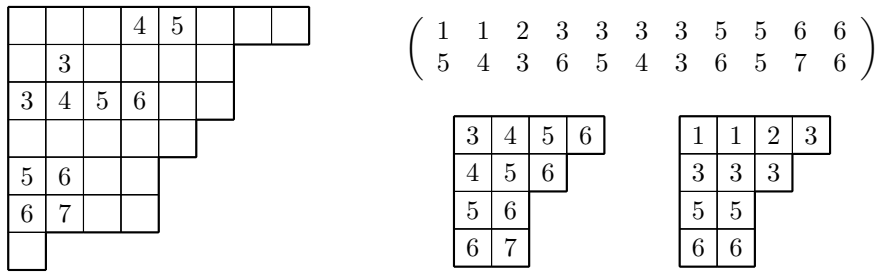
		It is known (and easy to check) that this gives a bijection between $\RP(\sigma)$ and the set of \Dfn{compatible sequences} $\CS( \sigma )$, defined by S.~Billey, W.~Jockush and R.~Stanley in \cite{BJS1993} as the set of all $2 \times \ell$ arrays of the form $\binom{a_1,\ldots,a_\ell}{b_1,\ldots,b_\ell}$ satisfying the following properties:
		\begin{enumerate}
			\item $a_1 \le a_2 \le \cdots \le a_{\ell}$,
			\item if $a_i = a_{i+1}$, then $b_i > b_{i+1}$,
			\item $b_1 b_2 \cdots b_{\ell}$ is a reduced word for $\sigma$, where $i$ denotes the simple transposition $s_i = (i,i+1)$, and
			\item $a_i \leq b_i$.
		\end{enumerate}
		One can see from the definition that a compatible sequence $t$ for $\sigma$ can be written as the concatenation $t = t_1 \cdots t_m$, where $t_i = \binom{i^{|w_i|}}{w_i}$, and $w_i$ is decreasing. Observe that $\sigma$ fixes all $j \leq k$ and thus, every letter in $w_i$ is larger than or equal to $\max(i,k)$.

		Define a map $\CS(\sigma) \rightarrow \FT(\mu,k)$ as follows. Let $t \in \CS(\sigma)$ be a compatible sequence for $\sigma$. Insert the letters of the word formed by the bottom row of $t$ using column Edelman--Greene insertion \cite{EG1987} into a tableau, while recording the corresponding letters from the first row. This produces an insertion tableau $P(t)$ and a recording tableau $Q(t)$. The image in $\FT(\mu,k)$ is now defined to be $Q(t)$. To prove that this is a well defined bijection, we need two preliminary lemmas (see \cite[Section 4]{Len2004}).
		\begin{lemma}\label{lem:insertion}
			All insertion tableaux $P(w)$ for reduced words $w$ of $\sigma = \sigma_k(\lambda)$ are equal. The shapes of $P(t)$ and $Q(t)$ are given by $\mu = \D(\sigma)$.
		\end{lemma}
		\begin{proof}
			Since $\mu$ is a Ferrers shape where the north-west corner is located in box $(k+1,k+1)$, $P(t)$ only depends on $\sigma$ and not on the actual compatible sequence (see \cite[Section~2.8.3]{Man2001}). Moreover, the labelling $i+j-k-1$ for $(i,j) \in \mu$ gives a reduced expression for $\sigma$ which column inserts into itself (see \cite[Remark~2.1.9]{Man2001}). Therefore, the shapes of $P(t)$ and $Q(t)$ are both given by $\mu$.
		\end{proof}
		\begin{lemma}\label{le:rows}
			Let $\sigma = \sigma_k(\lambda)$ for a Ferrers shape $\lambda$ and let $t = t_1 t_2 \cdots t_m$ be a compatible sequence for $\sigma$. Every letter in $w_i \cdots w_m$ is strictly larger than $j$ if and only if every letter in the first $j-k$ rows of $Q(t)$ is strictly less than $i$.
		\end{lemma}
		\begin{proof}
			Every letter in $w_i \cdots w_m$ is strictly larger than $j$ if and only if all of the occurrences of $1,2,\ldots, j$ in $w = w_1 \cdots w_m$ appear in $w_1 \cdots w_{i-1}$. Since $w$ is a reduced word for $\sigma$, this is equivalent to saying that the first $j$ letters of the permutation given by $w_1 \cdots w_{i-1}$ are the same as those in $\sigma$, when written in one line notation. Since $\mu = \D(\sigma)$ is a Ferrers shape where the north-west corner is located at the box $(k+1,k+1)$, this is equivalent to saying that every entry on the first $j-k$ rows of $Q(t)$ and of $Q(t_1 \cdots t_{i-1})$ coincide. As $Q(t_1 \cdots t_{i-1})$ contains only letters strictly smaller than $i$, the result follows.
		\end{proof}
		Putting the connections between reduced pipe dreams, compatible sequences and flagged tableaux together, we obtain the following theorem.
		\begin{theorem}
			Let $\sigma = \sigma_k(\lambda)$ for a Ferrers shape $\lambda$, and let $\mu = \D(\sigma)$. The map sending $D$ in $\RP(\sigma)$ to the recording tableau of the reading biword of $D$ is a bijection between $\RP(\sigma)$ and $\FT(\mu,k)$.
		\end{theorem}
		\begin{proof}
			It is left to show that the map sending a compatible sequence $t$ to $Q(t)$ is a well defined bijection between $\CS(\sigma)$ and $\FT(\mu,k)$. Let $t \in \CS(\sigma)$. By Lemma~\ref{lem:insertion}, $Q(t)$ has shape $\mu$, and by Lemma \ref{le:rows} with $i = \ell+k+1$ and $j = \ell+k$, every letter in row $\ell$ in $Q(t)$ is less than or equal to $\ell+k$. Thus, $Q(t)$ is indeed a $k$-flagged tableau. Furthermore, the construction is bijective, since Edelman--Greene insertion can be inverted to obtain $t$.
		\end{proof}
		An example of the bijection can be seen in Figure \ref{fig:biword}.
		\subsection{A cyclic action on flagged tableaux}\label{sec:cyclic action}
		In this subsection we define a cyclic action on $k$-flagged tableaux. The \Dfn{flagged promotion} $\rho(Q)$ of a $k$-flagged tableau $Q$ is defined as follows.
			\begin{enumerate}
				\item[(i)] Delete all the instances of the letter $1$,
				\item[(ii)] apply jeu de taquin to the remaining entries,
				\item[(iii)] subtract $1$ from all the entries,
				\item[(iv)] label each empty box on row $i$ with $i+k$.
		\end{enumerate}
		One can easily see that $\rho(Q)$ is indeed a $k$-flagged tableau, since the empty boxes after step (iii) must form a horizontal strip, which means there is at most one empty box per column. Furthermore, as every box gets moved at most up by one row, and at the end one subtracts $1$ from all the entries, the tableau obtained after step (iii) is $k$-flagged as well. The argument is finalized with the observation that if one adds a horizontal strip in which every box gets added its maximum possible value, the tableau is still $k$-flagged, since the row-weakness is assured by the maximality of the value of the entries on each row, and the column-strictness is assured by the fact that the entries in row $i-1$ are all strictly less than the maximal value on row~$i$.
		\subsection{From flagged tableaux to fans of paths and south-east fillings}
			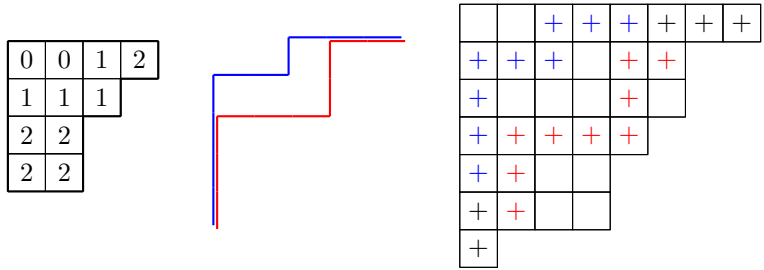
\begin{figure}
				$\begin{array}{ccccc}
					\begin{tikzpicture}[scale=1]
						\boxcollection{0.5}{(-2,2)}{
							2/2/$0$,3/2/$0$,4/2/$1$,5/2/$2$,
							2/3/$1$,3/3/$1$,4/3/$1$,
							2/4/$2$,3/4/$2$,
							2/5/$2$,3/5/$2$}
						\latticepath{0.5}{(-1,-0.5)}{thick}{
							1/0/black,1/0/black,0/1/black,0/1/black,1/0/black,0/1/black,1/0/black,0/1/black,-4/0/black,0/-4/black}
						\content{0.5}{(0,-1.15)}{
							0/0/}
					\end{tikzpicture}
					&&
					\begin{tikzpicture}[scale=1]
						\latticepath{0.5}{(-1.05,-0.5)}{thick}{
							0/1/red,0/1/red,0/1/red,1/0/red,1/0/red,1/0/red,0/1/red,0/1/red,1/0/red,1/0/red}
						\latticepath{0.5}{(-1.1,-0.45)}{thick}{
							0/1/blue,0/1/blue,0/1/blue,0/1/blue,1/0/blue,1/0/blue,0/1/blue,1/0/blue,1/0/blue,1/0/blue}
						\content{0.5}{(0,-0.65)}{
							0/0/}
					\end{tikzpicture}
					&&
					\begin{tikzpicture}[scale=1]
						\boxcollection{0.5}{(0,0)}{
							0/0/,1/0/,2/0/$\color{blue}+$,3/0/$\color{blue}+$,4/0/$\color{blue}+$,5/0/$\color{black}+$,6/0/$\color{black}+$,7/0/$\color{black}+$,
							0/1/$\color{blue}+$,1/1/$\color{blue}+$,2/1/$\color{blue}+$,3/1/,4/1/$\color{red}+$,5/1/$\color{red}+$,
							0/2/$\color{blue}+$,1/2/,2/2/,3/2/,4/2/$\color{red}+$,5/2/,
							0/3/$\color{blue}+$,1/3/$\color{red}+$,2/3/$\color{red}+$,3/3/$\color{red}+$,4/3/$\color{red}+$,
							0/4/$\color{blue}+$,1/4/$\color{red}+$,2/4/,3/4/,
							0/5/$\color{black}+$,1/5/$\color{red}+$,2/5/,3/5/,
							0/6/$\color{black}+$}
					\end{tikzpicture}
				\end{array}$
			\caption{Reverse plane partition corresponding to Figure \ref{fig:biword} and its corresponding $2$-fan of paths.}
			\label{fig:planepartitionpaths}
		\end{figure}
		We proceed as in \cite{FK1997} to obtain a reverse plane partition of height $k$ from a $k$-flagged tableau. Let $\lambda$ be a Ferrers shape and let $\mu = \D( \sigma_k( \lambda ) )$. Since every entry in row $i$ of a $k$-flagged tableau of shape $\mu$ is less than or equal to $i+k$ and greater than or equal to $i$ (as the tableau is semistandard), one can subtract $i$ from all the entries in row $i$, for all rows, and obtain a reverse plane partition of height $k$ and shape $\mu$, or equivalently, a $k$-fan of noncrossing north-east paths inside $\mu$. To obtain a bijection between $k$-flagged tableaux of shape $\mu$ and the set $\Fse(\lambda,k)$ of $k$-south-east fillings of the shape $\lambda$, one lifts the $i$-th path from the bottom by $i-1$ and turns it into a path of $+$'s inside $\lambda$. See Figure~\ref{fig:planepartitionpaths} for an example; the {\color{red}red} marks come from the red path, the {\color{blue}blue} from the blue path, and the additional black marks are contained in any $2$-south-east filling.

		Putting the described bijections together, we obtain Theorem~\ref{th:2}.
		\begin{theorem}
			Let $\lambda$ be a Ferrers shape. The composition of the described maps is a bijection between $\Fne(\lambda)$ and $\Fse(\lambda)$.
		\end{theorem}

		As mentioned in the introduction, $k$-triangulations of the $n$-gon can be seen as $k$-north-east fillings of the staircase shape $(n-1,\ldots,2,1)$, and $k$-fans of Dyck paths of length $2(n-2k)$ can be seen as $k$-south-east fillings of the same staircase (see e.g. \cite{Kra2006, Rub2006}). Thus, we obtain Theorem~\ref{th:1}. See Figure~\ref{fig:trianfan} for an example.
		\begin{corollary}
			In the case where $\lambda$ is the staircase shape $(n-1,\ldots,2,1)$, the described map is a bijection between $k$-triangulations of the $n$-gon and $k$-fans of noncrossing Dyck paths of length $2(n-2k)$.
		\end{corollary}
		\begin{figure}
			\newcounter{intege}
			$\begin{array}{c}
			  \begin{tikzpicture}[scale=0.9]
				\coordinate (allo) at (0,0);
				\ngon{8}{(allo)}{2}
						{1/2/verylightgrey,2/3/verylightgrey,3/4/verylightgrey,4/5/verylightgrey,5/6/verylightgrey,6/7/verylightgrey,7/8/verylightgrey,8/1/verylightgrey,1/3/verylightgrey,2/4/verylightgrey,3/5/verylightgrey,4/6/verylightgrey,5/7/verylightgrey,6/8/verylightgrey,7/1/verylightgrey,8/2/verylightgrey,
						1/4/black,1/6/black,3/6/black,3/7/black,3/8/black,4/7/black}
					{6}{1,2,3,4,5,6,7,8}{first}{90}
			\end{tikzpicture}
			\quad
			\begin{tikzpicture}[scale=0.9]
				\boxcollection{0.5}{(0,0)}{
					0/0/$\color{lightgrey}+$,1/0/$\color{lightgrey}+$,2/0/$+$,3/0/,4/0/,5/0/$\color{lightgrey}+$,6/0/$\color{lightgrey}+$,
					0/1/$\color{lightgrey}+$,1/1/,2/1/$+$,3/1/$+$,4/1/$\color{lightgrey}+$,5/1/$\color{lightgrey}+$,
					0/2/$+$,1/2/,2/2/$+$,3/2/$\color{lightgrey}+$,4/2/$\color{lightgrey}+$,
					0/3/,1/3/,2/3/$\color{lightgrey}+$,3/3/$\color{lightgrey}+$,
					0/4/$+$,1/4/$\color{lightgrey}+$,2/4/$\color{lightgrey}+$,
					0/5/$\color{lightgrey}+$,1/5/$\color{lightgrey}+$,
					0/6/$\color{lightgrey}+$}
				\content{0.5}{(0,1)}{
					0/0/$1$,1/0/$2$,2/0/$3$,3/0/$4$,4/0/$5$,5/0/$6$,6/0/$7$,
					-1/1/$8$,-1/2/$7$,-1/3/$6$,-1/4/$5$,-1/5/$4$,-1/6/$3$,-1/7/$2$}
				\latticepath{0.5}{(0,-3)}{thick}{
					1/0/black,0/1/black,1/0/black,0/1/black,1/0/black,0/1/black,1/0/black,0/1/black,1/0/black,0/1/black,1/0/black,0/1/black,1/0/black,0/1/black,-7/0/black,0/-7/black}
				\content{0.5}{(0,-3)}{
					0/0/}
			\end{tikzpicture}
			\quad
			\begin{tikzpicture}[scale=0.9]
				\tpipedream{0.5}{(0,0)}{
					0/0/darkgrey/darkgrey,1/0/darkgrey/darkgrey,2/0/darkgrey/darkgrey,5/0/darkgrey/darkgrey,6/0/darkgrey/white,
					0/1/darkgrey/darkgrey,2/1/darkgrey/darkgrey,3/1/darkgrey/darkgrey,4/1/darkgrey/darkgrey,5/1/darkgrey/white,
					0/2/darkgrey/darkgrey,2/2/darkgrey/darkgrey,3/2/darkgrey/darkgrey,4/2/darkgrey/white,
					2/3/darkgrey/darkgrey,3/3/darkgrey/white,
					0/4/darkgrey/darkgrey,1/4/darkgrey/darkgrey,2/4/darkgrey/white,
					0/5/darkgrey/darkgrey,1/5/darkgrey/white,
					0/6/darkgrey/white}
				\cpipedream{0.5}{(0,0)}{
					3/0/darkgrey/darkgrey,4/0/darkgrey/darkgrey,
					1/1/darkgrey/darkgrey,
					1/2/darkgrey/darkgrey,
					0/3/darkgrey/darkgrey,1/3/darkgrey/darkgrey}
				\content{0.5}{(0,1)}{
					0/0/$1$,1/0/$2$,2/0/$3$,3/0/$4$,4/0/$5$,5/0/$6$,6/0/$7$,
                    -1/1/$1$,-1/2/$2$,-1/3/$6$,-1/4/$5$,-1/5/$4$,-1/6/$3$,-1/7/$7$,
                     3/1/$\color{brown}s_4$,4/1/$\color{brown}s_5$,
                     1/2/$\color{green}s_3$,
                     1/3/$\color{blue}s_4$,
                     0/4/$\color{red}s_4$,1/4/$\color{red}s_5$}
				\content{0.5}{(0,-2.9)}{
					0/0/}
			\end{tikzpicture}
			\\[10pt]

\begin{array}{c}
\vspace*{-65pt}\\
\scalebox{.8}{$ \left(
\begin{array}{ccccccccccccccc}
\color{brown}1 & \color{brown}1 & \color{green}2 & \color{blue}3 & \color{red}4 & \color{red}4\\
   \color{brown}s_5 & \color{brown}s_4 &
   \color{green}s_3 &
   \color{blue}s_4 &
   \color{red}s_5 & \color{red}s_4
\end{array}
\right) $}
 \\ \\
   \begin{array}{cc}
             \begin{tikzpicture}[scale=0.9]
                 \boxcollection{0.5}{(0,0)}{
                     0/0/$3$,1/0/$4$,2/0/$5$,
                     0/1/$4$,1/1/$5$,
                     0/2/$5$}
                 \latticepath{0.5}{(0,0.5)}{thick}{
                     3/0/black,0/-1/black,-1/0/black,0/-1/black,-1/0/black,0/-1/black,-1/0/black,0/3/black}
                 \content{0.5}{(0,-1.65)}{
                     0/0/}
             \end{tikzpicture}
&
            \begin{tikzpicture}[scale=0.9]
                \boxcollection{0.5}{(0,0)}{
                    0/0/$1$,1/0/$1$,2/0/$2$,
                    0/1/$3$,1/1/$4$,
                    0/2/$4$}
                \latticepath{0.5}{(0,0.5)}{thick}{
                    3/0/black,0/-1/black,-1/0/black,0/-1/black,-1/0/black,0/-1/black,-1/0/black,0/3/black}
                \content{0.5}{(0,-1.65)}{
                    0/0/}
            \end{tikzpicture}
\end{array}
\end{array}
\qquad
				\begin{tikzpicture}[scale=0.9]
				\content{0.5}{(0,0)}{
					0/0/$0$,1/0/$0$,2/0/$1$,
					0/1/$1$,1/1/$2$,
					0/2/$1$}
				\latticepath{0.5}{(-0.025,-2.1)}{thick}{
					0/3.2/blue,2/0/blue,0/1/blue,2.2/0/blue,-4.2/-4.2/black}
				\latticepath{0.5}{(0.025,-2.05)}{thick}{
					0/1/red,1/0/red,0/2/red,2/0/red,0/1/red,1/0/red}
				\content{0.5}{(-0.25,-2.15)}{
					0/0/{\scriptsize (0,0)}}
				\content{0.5}{(1.85,0.5)}{
					0/0/{\scriptsize (4,4)}}
			\end{tikzpicture}
			\qquad
			\begin{tikzpicture}[scale=0.9]
				\boxcollection{0.5}{(0,0)}{
					0/0/,1/0/,2/0/$\color{blue}+$,3/0/$\color{blue}+$,4/0/$\color{blue}+$,5/0/$\color{black}+$,6/0/$\color{black}+$,
					0/1/$\color{blue}+$,1/1/$\color{blue}+$,2/1/$\color{blue}+$,3/1/,4/1/$\color{red}+$,5/1/$\color{red}+$,
					0/2/$\color{blue}+$,1/2/,2/2/$\color{red}+$,3/2/$\color{red}+$,4/2/$\color{red}+$,
					0/3/$\color{blue}+$,1/3/,2/3/$\color{red}+$,3/3/,
					0/4/$\color{blue}+$,1/4/$\color{red}+$,2/4/$\color{red}+$,
					0/5/$\color{black}+$,1/5/$\color{red}+$,
					0/6/$\color{black}+$}
				\latticepath{0.5}{(0,-3)}{thick}{
					1/0/black,0/1/black,1/0/black,0/1/black,1/0/black,0/1/black,1/0/black,0/1/black,1/0/black,0/1/black,1/0/black,0/1/black,1/0/black,0/1/black,-7/0/black,0/-7/black}
			 \end{tikzpicture}
			\end{array}$
	\caption{An example of all the steps in the bijection: a $2$-triangulation of the $8$-gon, the $2$-north-east filling of $(7,\ldots,1)$, the pipe dream for $[1,2,6,5,4,3,7,8]$, the compatible sequence, the insertion and recording ($2$-flagged) tableau of shape $(3,2,1)$, the reverse plane partition of shape $(3,2,1)$, the $2$-fan of Dyck paths of length $8$, and finally the $2$-south-east filling of $(7,\ldots,1)$.}
	\label{fig:trianfan}
\end{figure}
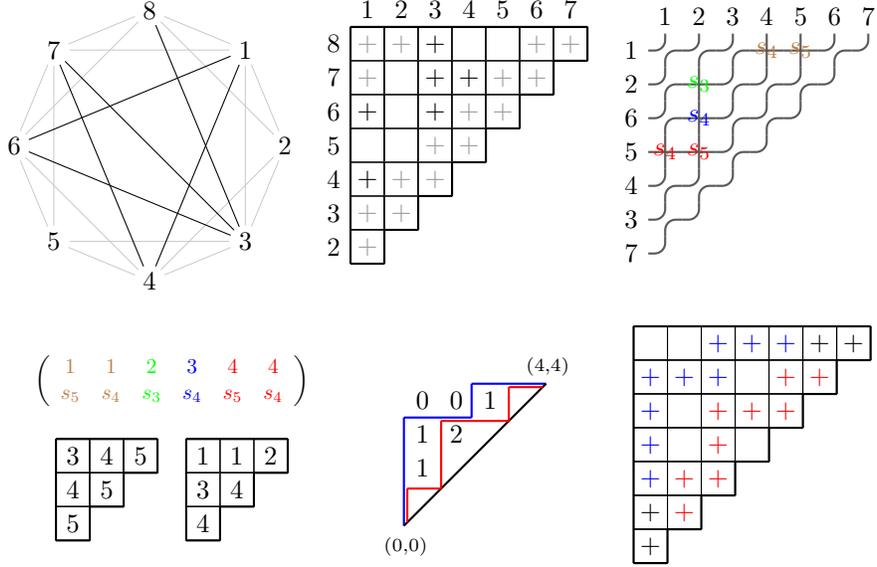

	\section{Properties of north-east fillings and $k$-triangulations}
		Using Theorem~\ref{th:moon fillings-pipe dreams}, we obtain several properties of $k$-north-east fillings of moon polyominoes and of Ferrers shapes and $k$-triangulations in particular. Some of them where already known while others where only conjectured.
		
		The first property was proved in the case of stack polyominoes by J.~Jonsson in \cite[Theorem~10]{J2005}. It follows immediately from Theorem~\ref{th:fillings-pipe dreams}.
		\begin{corollary}
			Every $k$-north-east filling of a moon polyomino $M$ contains $i$ many boxes where $i$ equals the total number of boxes in $M$ minus the length of $\sigma_k(M)$. In particular $i$ equals the number of boxes in the first $k$ rows and columns in the case of Ferrers shapes.
		\end{corollary}

		The second property is part of the main theorem in \cite[Theorem~1.4(i)]{PS2009} and concerns the star property as described as well in \cite{Stu2010}; for the notion used here, we refer as well to the latter.
		\begin{corollary}
			Every $k$-triangulation of the $n$-gon consists of exactly $n-2k$ $k$-stars.
		\end{corollary}
		\begin{proof}
			This follows from the description of $k$-triangulations in terms of $k$-north-east fillings of the staircase shape. As in this case $\sigma$ is given by $\1{k} \times [n-2k,\ldots,1]$, we obtain $2k$ outer pipe, as well as $n-2k$ inner pipes connecting $i$ with $i$ for $k < i \leq n-k$, and which contains exactly $2k+1$ turns. See Figure~\ref{fig:trianfan} for an example. This is exactly the star property in \cite{Stu2010} and thus completes the proof.
		\end{proof}
		Using the description of mutations for $k$-triangulations in Section~\ref{sec:mutation}, one can also describe the mutation of a facet in the simplicial complex
		$$\Delta_{n,k} := \Delta(\1{k} \times [n-2k,\ldots,1]).$$
		This mutation corresponds to removing a diagonal in a $k$-triangulation and replacing it by the unique other diagonal which gives a $k$-triangulation. This operation is called \emph{flip} in \cite[Theorem~1.4(iii)]{PS2009}.
		\begin{corollary}
			A facet $F$ in the simplicial complex $\Delta_{n,k}$ can be mutated at any vertex $d = (i,j) \in F$ for which $k < | i - j | < n - k$.
		\end{corollary}
		\begin{proof}
			The inversions of $\sigma = \1{k} \times [n-2k,\ldots,1]$ are given by all $(i,j)$ for which $k < i < j \leq n-k$. Thus, all $n-2k$ inner pipes in $D \in \RP(\sigma)$ mutually cross. It follows from Section~\ref{sec:mutation} that the facet corresponding to $D$ can be mutated at any vertex $(i,j)$ for which $k < | i - j | < n - k$.
		\end{proof}

			The next property of the constructed bijection will allow us to obtain a refined counting of $k$-triangulations, as conjectured by C.~Nicolas \cite{Nic2009}. Note that a diagonal $(i,j)$ for which $|i-j|\leq k$ or $|i-j|\geq n-k$ is contained in every $k$-triangulation of the $n$-gon. Thus, we define the degree of a vertex $i$ as the number of vertices $j$ adjacent to $i$ for which $k < |i-j| < n-k$.
			\begin{theorem}\label{th:degree}
				The degree of vertex $1$ in a $k$-triangulation is equal to the number of touching points of the lowermost Dyck path of its corresponding $k$-fan of Dyck paths with the main diagonal. Furthermore, each edge $(1,j)$ corresponds to the touching point with coordinates $(j-k-1,j-k-1)$.
			\end{theorem}
			\begin{proof}
				Let $t$ be the compatible sequence corresponding to a $k$-triangulation $T$, decomposed into $t_1 \cdots t_m$, where $t_ i = \binom{i^{|w_i|}}{w_i}$ as in the definition. By construction, there is a diagonal $(1,j)$ if and only if the column $\binom{n+1-j}{n+1-j}$ does not appear in $t$. By virtue of this, and the fact that each letter in $w_i$ is larger than or equal to $i$, we have that every letter in $w_{n+1-j} \cdots w_m$ is strictly larger than $n+1-j$. By Lemma \ref{le:rows} with $i=n-j+1$, every letter in the first $n+1-j-k$ rows of $Q(T)$ is smaller than or equal to $n-j$. In particular, every letter in row $n+1-j-k$ of the corresponding reverse plane partition is strictly smaller than $k$. By construction, this implies that the lowermost path touches the diagonal at the point $(j-k-1,j-k-1)$. The converse follows clearly from the argument.
			\end{proof}
			Figure~\ref{fig:trianfan} shows an example for $k=2$, where vertex $1$ is connected to vertices $4$ and $6$, and the ({\color{red}red}) lowermost Dyck path touches the diagonal at positions
			\begin{align*}
			 (1,1) &= (4-2-1,4-2-1) \text{ and}\\
			 (3,3) &= (6-2-1,6-2-1).
			\end{align*}
			As described in \cite{Nic2009}, we use this theorem to prove Conjecture~2 therein. For an explicit expression for the determinant, we refer to \cite[Theorem~4]{Kra20062}.
			\begin{corollary}\label{cor:determinant}
				The number of $k$-triangulations of a convex $n$-gon having degree $d$ in a given vertex is given by the determinantal expression
				$$\det \left( \begin{array}{cccc}
													\Cat_{n-2k} & \cdots & \Cat_{n-k-2} & B_{n-k-1}^k(d)\\
			    	            	\vdots & \ddots & \vdots &\vdots\\
													\Cat_{n-k-1} & \cdots & \Cat_{n-3} & B_{n-2}^k(d)
			        	       \end{array}
				         \right),
				$$
				where $\Cat_\ell$ is the usual Catalan number, and where $B_\ell^k(d) = \frac{2k+d-3}{\ell} \binom{2\ell-2k-d+2}{\ell-1}$.
			\end{corollary}

		\subsection{Rotation of the $n$-gon and a CSP for flagged tableaux}
			There is a natural cyclic action $\rho$ on $k$-triangulations given by rotating the vertex labels in the $n$-gon counterclockwise. The following conjecture is due to V.~Reiner~\cite{Re2009}.
			\begin{conjecture}[V.~Reiner]\label{con:CSP1}
				Let $\lambda$ be the staircase shape $(n-1,\ldots,2,1)$ and let $k$ be a positive integer. The triple
				$$\Big( \Fne(\lambda,k), \langle \rho \rangle, F(q) \Big)$$
				exhibits the cyclic sieving phenomenon (CSP) as described in \cite{RSW2004}.
			\end{conjecture}
			We can describe the cyclic action on $k$-triangulations induced by rotation in terms of the cyclic action on flagged tableaux as defined in Section~\ref{sec:cyclic action}.
			\begin{theorem}
				The constructed bijection maps the cyclic action on $k$-triangulations to the cyclic action given by flagged promotion on flagged tableaux.
			\end{theorem}
			\begin{proof}
				Let $T$ be a triangulation with compatible sequence $t$, and recording tableau $Q = Q(t)$. First, we describe the compatible sequence $\rho(t)$ of $\rho(T)$, then we describe the compatible sequence for the flagged promotion $\rho(Q)$, and finally we show that they are the same.

				Note that any non-edge $(i,j)$ in $T$, with $i < j$, gets encoded in the compatible sequence as a column $\binom{n+1-j}{n+i-j}$. In particular, the non-edges $(i,n)$ get encoded as $\binom{1}{i}$. For $j < n$, the switch from the non-edge $(i,j)$ in $T$ to $(i+1, j+1)$ in $\rho(T)$ corresponds to turning the column $\binom{n+1-j}{n+i-j}$ in $t$ into the column $\binom{n-j}{n+i-j}$ in $\rho(t)$. Likewise, the switch from the non-edge $(i,n)$ in $T$ to $(1,i+1)$ in $\rho(T)$ corresponds to turning the column $\binom{1}{i}$ into the column $\binom{n-i}{n-i}$, and placing it in the right place to make sure $\rho(t)$ is a compatible sequence.

				Let $\tilde{Q}$ be the ordinary promotion of $Q$. It is well known by the relationship between promotion and Edelman--Greene insertion, that $\tilde{Q}$ is the recording tableau of the biword obtained by turning each column of the form $\binom{i}{j}$ into the column $\binom{i-1}{j}$, for $2 \le i \le n$, and turning each column of the form $\binom{1}{j}$ into one of the form $\binom{n}{n-j}$, and placing them at the end of the compatible sequence in reverse order. This is the same transformation as described in the above paragraph, except for columns $\binom{1}{j}$ which got mapped into columns $\binom{n-j}{n-j}$ instead. We now proceed to slide these columns of the form $\binom{n}{n-j}$ towards the position of the column $\binom{n-j}{n-j}$ one by one, starting from the left while adjusting the element in the top row accordingly, and show that this transforms $\tilde{Q}$ into $\rho(Q)$. Notice that every letter in the bottom row between these two columns is greater than or equal to $n-j+2$, which means they all commute with $n-j$. Thus, the bottom row is still a reduced word for $\sigma$. Lemma~\ref{lem:insertion} implies that that all reduced expressions for $\sigma$ have the same insertion tableau. Since we now record $n-j$ instead of $n$ for those columns, this procedure transforms $\tilde{Q}$ into $\rho(Q)$.
			\end{proof}
			Using this connection, we obtain the following corollary.
			\begin{corollary}
				Conjecture~\ref{con:CSP2} is equivalent to Conjecture~\ref{con:CSP1}.
			\end{corollary}

		\section{A positivity result for Schubert polynomials}

			In this section we use the results about moon polyominoes in Section~\ref{sec:moon} to obtain a property of Schubert polynomials. It was shown in \cite{FK1996} that Schubert polynomials are a generating series for pipe dreams, more precisely, for a permutation $\sigma$,
			\[
			\mathfrak{S}_{\sigma} (x_1, \ldots, x_n) = \sum_{D \in \RP(\sigma)} \prod_{(i,j) \in D} x_i.
			\]

			We obtain the following theorem and thus Corollary~\ref{cor:3}.
			\begin{theorem}
				Let $S$ be a stack polyomino, let $\lambda$ be the associated Ferrers shape and let $k$ be a positive integer. Then $$\mathfrak{S}_{\sigma_k(S)}(x_1,x_2,\ldots) - \mathfrak{S}_{\sigma_k( \lambda )}(x_1,x_2,\ldots)$$
				is monomial positive. In particular, $\mathfrak{S}_{\sigma_k(S)}(1,1,\ldots)$ is greater than or equal to the number of $k$-flagged tableaux of shape $\lambda$.
			\end{theorem}
			\begin{proof}
				This follows from Corollary~\ref{cor:2}.
			\end{proof}

		\section{Acknowledgements}
			The authors would like to thank Christian Krattenthaler for pointing to a conjecture in \cite{Nic2009} proved in Corollary~\ref{cor:determinant}. The first author would also like to thank Sergey Fomin and David Speyer for helpful discussions.

	\bibliographystyle{amsalpha} 
	\bibliography{../../../bibliography}
\end{document}